\documentclass[11pt,reqno]{amsart}

\usepackage[T1]{fontenc}
\usepackage[utf8]{inputenc}
\usepackage{mathpazo}
\usepackage[scaled=.95]{helvet}
\usepackage{microtype}
\usepackage[margin=1.10in]{geometry}
\usepackage{amsmath,amssymb,mathtools,mathrsfs}
\usepackage{xcolor}
\usepackage{enumitem}
\usepackage{hyperref}
\usepackage[nameinlink,noabbrev]{cleveref}

\hypersetup{
  colorlinks=true,
  linkcolor=blue!55!black,
  citecolor=blue!55!black,
  urlcolor=blue!55!black,
  pdfauthor={Aelson Sobral, Eduardo V. Teixeira and Jose Miguel Urbano},
  pdftitle={Qualitative-to-quantitative regularity for degenerate equations}
}

\numberwithin{equation}{section}
\theoremstyle{plain}
\newtheorem{theorem}{Theorem}[section]

\newtheorem{lemma}[theorem]{Lemma}
\newtheorem{corollary}[theorem]{Corollary}

\theoremstyle{definition}

\theoremstyle{remark}
\newtheorem{remark}[theorem]{Remark}

\crefname{hypothesis}{hypothesis}{hypotheses}
\Crefname{hypothesis}{Hypothesis}{Hypotheses}

\newcommand{\R}{\mathbb R}
\newcommand{\N}{\mathbb N}
\newcommand{\Rn}{\mathbb R^n}
\newcommand{\Sn}{\mathbb S^{n-1}}
\newcommand{\norm}[1]{\left\lVert #1\right\rVert}
\newcommand{\abs}[1]{\left\lvert #1\right\rvert}
\newcommand{\osc}{\operatorname*{osc}}
\newcommand{\diver}{\operatorname{div}}
\newcommand{\Fcal}{\mathcal F}

\newcommand{\Kcal}{\mathcal K}

\setlist[enumerate]{leftmargin=2.2em,itemsep=4pt,topsep=4pt}

\setlist[itemize]{leftmargin=2em,itemsep=3pt,topsep=4pt}

\title[Qualitative-to-quantitative regularity]{Qualitative-to-quantitative regularity for degenerate equations}

\author[A. Sobral]{Aelson Sobral}
\address{Applied Mathematics and Computational Sciences (AMCS), Computer, Electrical and Mathematical Sciences and Engineering Division (CEMSE), King Abdullah University of Science and Technology (KAUST), Thuwal, 23955-6900, Kingdom of Saudi Arabia}
\email{aelson.sobral@kaust.edu.sa}

\author[E. V. Teixeira]{Eduardo V. Teixeira}
\address{Department of Mathematics, Oklahoma State University, 401 Mathematical
Sciences, Stillwater, OK 74078, USA}
\email{eduardo.teixeira@okstate.edu} 

\author[J.M. Urbano]{Jos\'{e} Miguel Urbano}
\address{Applied Mathematics and Computational Sciences (AMCS), Computer, Electrical and Mathematical Sciences and Engineering Division (CEMSE), King Abdullah University of Science and Technology (KAUST), Thuwal, 23955-6900, Kingdom of Saudi Arabia; and CMUC, Department of Mathematics, University of Coimbra, 3000-143 Coimbra, Portugal}
\email{miguel.urbano@kaust.edu.sa}

\date{}

\subjclass[2020]{Primary 35B65; Secondary 35J92, 35J94, 35R11, 35D40}

\keywords{Qualitative-to-quantitative regularity, $p$-Laplacian, infinity-Laplacian, nonlocal super\-conductivity}

\begin{document}

\begin{abstract}
We establish a principle that converts qualitative regularity into 
uniform interior estimates for degenerate equations. For the
$p$-Laplace equation, $p>2$, we prove that local $C^{1,\gamma}$
regularity of every solution, without any uniform bound on the
corresponding seminorms, implies uniform $C^{1,\beta}$ estimates
for every $0<\beta<\gamma$. Thus, the optimal qualitative and
quantitative gradient H\"older exponents coincide as suprema in
every dimension. In the plane, this yields estimates controlled
by the $L^p$ norm of the solution for every exponent below the
sharp Iwaniec--Manfredi threshold. Applied to Xu's recent qualitative $C^{1,1/3}$ regularity for planar infinity-harmonic functions, we obtain uniform interior
$C^{1,\beta}$ estimates for every $0<\beta<1/3$. For the nonlocal superconductivity
equation $|Du|(-\Delta)^s u=0$, with $1/2<s<1$, we prove that
pointwise differentiability of all solutions with integrable
tails implies uniform interior $C^{1,\alpha}$ estimates for
bounded solutions, for some $\alpha>0$. The proofs combine compactness
for shifted equations with an improvement-of-flatness iteration
at variable scales. 
\end{abstract}

\maketitle

\section{Introduction and main results}

We study the relation between qualitative regularity and uniform interior estimates for solutions of degenerate PDEs, in particular, for the homogeneous $p$-Laplace equation
\[
-\Delta_pu\coloneqq-\diver\bigl(\abs{Du}^{p-2}Du\bigr)=0,\qquad p>2.
\]
The question is whether knowing that every solution belongs locally to $C^{1,\gamma}$, without a common norm bound, suffices to obtain a quantitative estimate for every exponent below $\gamma$. Our main result gives an affirmative answer to this question. In the plane, this yields uniform H\"older gradient estimates for every exponent strictly below the sharp exponent of Iwaniec and Manfredi, which are key in the proof of the $C^{p^{\prime}}$-regularity conjecture.

Interior $C^{1,\alpha}$ regularity for $p$-harmonic functions is classical. The works of Uhlenbeck, Evans, DiBenedetto, Lewis, and Tolksdorf give quantitative estimates for some $\alpha=\alpha(n,p)>0$; see \cite{Uhlenbeck1977, Evans1982, DiBenedetto1983, Lewis1983, Tolksdorf1984}. The precise optimal exponent is unknown in general dimensions. In the plane, Iwaniec and Manfredi \cite{IwaniecManfredi1989} proved that every $p$-harmonic function belongs locally to $C^{1,\alpha_{\mathrm{IM}}(p)}$, where
\begin{equation}\label{eq:intro-IM-exponent}
\alpha_{\mathrm{IM}}(p)\coloneqq\frac16\left(\frac p{p-1}+\sqrt{1+\frac{14}{p-1}+\frac1{(p-1)^2}}\right),
\end{equation}
and showed that this exponent is sharp. Their proof uses the quasiregularity of the complex gradient and a hodograph transformation near critical points, and is not adaptable to higher dimensions.

The passage from this qualitative statement to a uniform estimate requires an additional argument. As observed by Ara\'ujo, Teixeira, and Urbano \cite[Section~2]{AraujoTeixeiraUrbano2017}, a universal bound for the $C^{1,\alpha_{\mathrm{IM}}(p)}$ norm is not derived in \cite{IwaniecManfredi1989}, and following the constants through the hodograph inversion is a delicate issue. Applying the estimates of Baernstein and Kovalev \cite{BaernsteinKovalev2005} to the complex gradient instead gives quantitative regularity with exponent
\begin{equation*}
\alpha_{\mathrm{BK}}(p)\coloneqq\frac1{2p}\left(-3-\frac1{p-1}+\sqrt{33+\frac{30}{p-1}+\frac1{(p-1)^2}}\right),
\end{equation*}
which satisfies
\begin{equation*}
\frac1{p-1}<\alpha_{\mathrm{BK}}(p)<\alpha_{\mathrm{IM}}(p).
\end{equation*}
This quantitative improvement over $1/(p-1)$ is a crucial ingredient in the proof of the planar $C^{p^{\prime}}$-regularity conjecture of \cite{AraujoTeixeiraUrbano2017}. The same authors established in \cite{AraujoTeixeiraUrbano2018} the $C^{p^{\prime}}$-regularity conjecture in higher dimensions for several special scenarios. The distinction between the qualitative and quantitative exponents also appears in the higher-order estimates of Balci, Diening, and Weimar \cite[Lemma~2.19 and Remark~2.20]{BalciDieningWeimar2020}.

More recently, Lee and Lee \cite{LeeLee2026} obtained quantitative gradient estimates beyond $\alpha_{\mathrm{BK}}(p)$ when $p>2$ is sufficiently close to $2$, and proved the $C^{p^{\prime}}$-regularity conjecture in every dimension in that range. The planar counterpart of our main result covers the full range $0<\beta<\alpha_{\mathrm{IM}}(p)$ for each fixed $p>2$. It uses the Iwaniec--Manfredi theorem only as qualitative input and does not require tracking the constants in its proof.

\subsection{Definitions and main result}

We assume $n\ge2$ and consider the $p$-Laplace equation in the degenerate case $p>2$. A function $u\in W^{1,p}_{\mathrm{loc}}(\Omega)$ is called $p$-harmonic in an open, bounded set $\Omega \subset \Rn$ if
\[
\int_\Omega \abs{Du}^{p-2}Du\cdot D\varphi\,dx=0, \qquad\text{for every }\varphi\in C_c^\infty(\Omega).
\]
We write $[Du]_{C^{0,\beta}(U)}$ for the H\"older seminorm of the gradient on $U$. Our main result converts qualitative gradient regularity into a uniform interior estimate.

\begin{theorem}\label{thm:intro-principle}
Let $\gamma\in(0,1]$, and suppose that every $p$-harmonic function in every ball of $\Rn$ belongs locally to $C^{1,\gamma}$. Then, for every $\beta\in(0,\gamma)$, there is $C=C(n,p,\gamma,\beta)<\infty$ such that every $p$-harmonic function $u$ in $B_1$ satisfies
\begin{equation}\label{eq:intro-main-estimate}
\norm{Du}_{L^\infty(B_{1/2})}+[Du]_{C^{0,\beta}(B_{1/2})}
\le C\norm{u}_{L^\infty(B_{3/4})}.
\end{equation}
The constant $C$ is independent of $u$.
\end{theorem}

The main application is the following planar estimate, obtained by taking $\gamma=\alpha_{\mathrm{IM}}(p)$.

\begin{corollary}\label{cor:intro-planar}
Let $u$ be $p$-harmonic in $B_1 \subset \R^2$. For every $0<\beta<\alpha_{\mathrm{IM}}(p)$, there is $C=C(p,\beta)<\infty$ such that
\begin{equation}\label{eq:intro-planar-Lp}
\norm{u}_{C^{1,\beta}(B_{1/2})}\le C(p,\beta)\norm{u}_{L^p(B_1)}.
\end{equation}
\end{corollary}

Together with the sharpness of the Iwaniec--Manfredi exponent, this identifies the optimal quantitative planar exponent as a supremum. The argument does not give the endpoint estimate at $\beta=\alpha_{\mathrm{IM}}(p)$. The constants involved are not explicit and may deteriorate as $\beta$ approaches the endpoint.

The same principle shows that the optimal qualitative and quantitative $p$-harmonic exponents agree as suprema in every dimension; see \Cref{cor:thresholds}. We state the estimates on the unit ball; their counterparts on arbitrary balls follow by translation and scaling.

\subsection{Contribution and proof strategy}

Our proof adapts the qualitative-to-quantitative compactness
method of Sobral and Teixeira \cite{ST}, developed for compact
families of uniformly elliptic equations preserved by their
natural rescalings. For the $p$-Laplace equation, affine subtraction generates a family of shifted equations, and the
argument must account for limits in which the shifts become
unbounded; see \cite{AttouchiParviainenRuosteenoja2017} for the
shifted framework. Uniform local H\"older estimates, independent
of the shift, provide compactness for normalized solutions.

After passing to a subsequence, either the shifts converge, or their lengths tend to infinity. In the first case, the limit is $p$-harmonic after adding an affine function. In the second case, the directions of the shifts converge to a unit vector $e$, and the limit solves a linear uniformly elliptic equation with the constant-coefficients matrix $I + (p-2)e\otimes e$. The first limit has the assumed qualitative regularity, while the second is smooth. Thus, every limiting function $h$ satisfies
\[
\norm{h-h(0)-Dh(0)\cdot x}_{L^\infty(B_r)}\le L_h r^{1+\gamma}
\]
for sufficiently small $r$, where $L_h$ may depend on $h$.

Fix $0<\beta<\gamma$. The factor $r^{\gamma-\beta}$ allows us to absorb $L_h$ by choosing a sufficiently small radius for each limiting function. The resulting strict affine approximation persists under small uniform perturbations. These neighborhoods cover the compact family of limits, so a finite subcovering provides finitely many admissible radii and affine slopes, independent of the solution and the shift.

At each step of the iteration, we choose one of these radii and subtract the corresponding affine approximation. The positive minimum of the selected radii allows estimates at the resulting scales to control all intermediate radii. The affine slopes converge, giving a uniform first-order expansion with remainder of order $1+\beta$. Applying the construction at every interior point and comparing the resulting expansions yields the uniform gradient estimate.

\subsection{Further main results}

The mechanism developed here yields two further principal contributions. Both rest on the same conceptual foundation, but each requires different adaptations of the compactness and rescaling argument to the structure of the pertaining equation. For the infinity-Laplace equation
\begin{equation*}
	\Delta_\infty u\coloneqq\langle D^2u\,Du,Du\rangle=0,
\end{equation*}
the planar estimates of Evans and Savin \cite{EvansSavin2008} provide $C^1$ compactness. Combining this compactness with Xu's qualitative $C^{1,1/3}$ regularity theorem and quantitative estimate near affine functions \cite[Theorems~1.1 and~8.2]{Xu2026} gives the following uniform interior estimate.

\begin{theorem}\label{thm:planar-critical-almost-endpoint}
For every $\beta\in(0,1/3)$ there is $C_\beta<\infty$ such that every infinity-harmonic function $u$ in $B_1\subset\R^2$ satisfies
\begin{equation}\label{eq:planar-uniform-gradient}
\norm{Du}_{L^\infty(B_{1/2})}+[Du]_{C^{0,\beta}(B_{1/2})}
\le C_\beta\osc_{B_{3/4}}u.
\end{equation}
The constant depends only on $\beta$.
\end{theorem}

The iteration uses only subtraction of constants and rescaling, while the normalized central gradient remains small. Once this gradient reaches a fixed positive threshold, the near-affine estimate controls all smaller scales. Together with Aronsson's example \cite{Aronsson1984}, the theorem identifies the optimal quantitative exponent as a supremum. 

The last application concerns the nonlocal equation
\begin{equation}\label{eq:scalar-nonlocal}
\abs{Du}(-\Delta)^su=0\quad\text{in }B_1,\qquad \frac12<s<1,
\end{equation}
associated with the superconductivity problem studied in \cite{AraujoSobral2026}. To formulate the qualitative assumption, we use the space
\begin{equation}\label{eq:tail-space}
L_s^1(\Rn)\coloneqq\left\{v\in L^1_{\mathrm{loc}}(\Rn):\norm v_{L_s^1(\Rn)}\coloneqq\int_{\Rn}\frac{\abs{v(y)}}{1+\abs{y}^{n+2s}}\,dy<\infty\right\}.
\end{equation}

\begin{theorem}\label{thm:nonlocal-conditional}
Let $s\in(1/2,1)$, and suppose that every viscosity solution $v\in C(B)\cap L_s^1(\Rn)$ of $\abs{Dv}(-\Delta)^sv=0$ in a ball $B$ is Fr\'echet differentiable at every point of $B$. Then, there are $\alpha=\alpha(n,s)\in(0,2s-1)$ and $C=C(n,s)<\infty$ such that every viscosity solution $u\in L^\infty(\Rn)\cap C(B_1)$ of \eqref{eq:scalar-nonlocal} satisfies
\begin{equation}\label{eq:nonlocal-conditional-estimate}
\norm{Du}_{L^\infty(B_{1/2})}+[Du]_{C^{0,\alpha}(B_{1/2})}\le C\norm u_{L^\infty(\Rn)}.
\end{equation}
The constants are independent of $u$.
\end{theorem}

The point of this application is to recover uniform estimates from differentiability alone, without using a quantitative gradient regularity theorem as input. The differentiability assumption includes unbounded solutions with integrable tails, since identifying limits of shifted solutions requires adding affine functions. After rescaling, the local Lipschitz estimates of \cite{DosPrazeresTopp2021} give locally uniform convergence along a subsequence on expanding balls. Together with a common polynomial growth bound of order strictly below $2s$, this yields strong convergence in $L_s^1(\Rn)$. The approximation lemma also ensures that this growth bound is preserved throughout the iteration.

\subsection{Paper organization} The paper is organized as follows. \Cref{sec:holder-principle} proves \Cref{thm:intro-principle} and \Cref{cor:intro-planar} and compares the optimal qualitative and quantitative $p$-harmonic exponents. \Cref{sec:infinity-critical-points} proves \Cref{thm:planar-critical-almost-endpoint}. Finally, \Cref{sec:nonlocal} establishes the nonlocal approximation lemma, including the growth bound at infinity, and proves \Cref{thm:nonlocal-conditional}.

\section{Uniform estimates for \texorpdfstring{$p$}{p}-harmonic functions}\label{sec:holder-principle}

In this section, we prove \Cref{thm:intro-principle} and derive the planar estimate in \Cref{cor:intro-planar}. We begin with the shifted equation generated by affine subtraction, then establish compactness and improvement of flatness before carrying out the iteration. We conclude by comparing the optimal qualitative and quantitative regularity exponents.

\subsection{The shifted normalized equation}

Throughout this section, $n\ge2$ and $p>2$. We consider the homogeneous $p$-Laplace equation
\begin{equation*}
  -\Delta_pu\coloneqq-\diver\bigl(\abs{Du}^{p-2}Du\bigr)=0.
\end{equation*}
Subtracting an affine function leads to a family of shifted equations. The estimates below must be uniform with respect to the shift. For $q,\xi\in\Rn$ and $X\in\operatorname{Sym}(n)$, define, when $\xi+q\ne0$,
\begin{equation}\label{eq:Fq}
  \Fcal_q(\xi,X)\coloneqq-\operatorname{tr}X-(p-2)\left\langle X\frac{\xi+q}{\abs{\xi+q}},\frac{\xi+q}{\abs{\xi+q}}\right\rangle.
\end{equation}

We use the viscosity formulation in which only test functions with a nonzero shifted gradient at the contact point are considered. More precisely, an upper semicontinuous function $v$ is a viscosity subsolution of $\Fcal_q(Dv,D^2v)=0$ in $\Omega$ if, whenever $\phi\in C^2$ touches $v$ from above at $x_0\in\Omega$ and $D\phi(x_0)+q\ne0$, one has
\begin{equation*}
  \Fcal_q(D\phi(x_0),D^2\phi(x_0))\le0.
\end{equation*}
A lower semicontinuous function is a viscosity supersolution if the reverse inequality holds for test functions touching from below with $D\phi(x_0)+q\ne0$. A continuous function is a viscosity solution if it is both a subsolution and a supersolution. No condition is imposed on test functions for which $D\phi(x_0)+q=0$.

For the homogeneous $p$-Laplace equation, this formulation is equivalent to $p$-harmonicity in the weak sense; see \cite[Definition~2.3, Remark~2.4, and Corollary~2.6]{JuutinenLindqvistManfredi2001} and the discussion following Definition~2.3 in \cite{JulinJuutinen2012}. For the shifted equation, the equivalence follows by adding the affine function $q\cdot x$. We record these equivalent formulations in the next lemma.

\begin{lemma}\label{lem:equivalence}
Let $q\in\Rn$ and $v\in C(\Omega)$. The following are equivalent:
\begin{enumerate}
  \item $v$ is a viscosity solution of $\Fcal_q(Dv,D^2v)=0$ in $\Omega$;
  \item $u(x)=v(x)+q\cdot x$ is a viscosity solution of the normalized
  homogeneous $p$-Laplace equation;
  \item $u(x)=v(x)+q\cdot x$ belongs to $W^{1,p}_{\mathrm{loc}}(\Omega)$ and is
  weakly $p$-harmonic;
  \item $v\in W^{1,p}_{\mathrm{loc}}(\Omega)$ and
  \begin{equation*}
    -\diver\bigl(\abs{Dv+q}^{p-2}(Dv+q)\bigr)=0
  \end{equation*}
  weakly in $\Omega$.
\end{enumerate}
\end{lemma}

\subsection{Limits of normalized shifted solutions}

Define
\begin{equation*}
  \Kcal_0 \coloneqq \left\{ v|_{\overline B_{1/2}} : \norm v_{L^\infty(B_1)}\le 1 \quad \textrm{and} \quad \Fcal_q(Dv,D^2v)=0 \ \ \text{in}\ B_1, \ \text{for some}\ q\in \Rn \right\},
\end{equation*}
and let
\begin{equation*}
\Kcal \coloneqq
\overline{\Kcal_0}^{\,C(\overline B_{1/2})}.
\end{equation*}
The next lemma contains the compactness statement and the only two possible limiting equations. The alternatives are not new; they were explicitly proved in \cite{AttouchiParviainenRuosteenoja2017}. We provide here a simpler proof.

\begin{lemma}\label{lem:profile-compactness}
The set $\Kcal$ is compact in $C(\overline B_{1/2})$, endowed with the uniform norm. For every $h\in\Kcal$, at least one of the following alternatives holds:
\begin{enumerate}
\item There is $q\in\Rn$ such that $h+q\cdot x$ is weakly $p$-harmonic in $B_{1/2}$.
\item There is $e\in\Sn$ such that
\begin{equation*}
  -\operatorname{tr}\bigl((I+(p-2)e\otimes e)D^2h\bigr)=0
  \quad\text{in }B_{1/2}.
\end{equation*}
\end{enumerate}
\end{lemma}

\begin{proof}
For $a\in\Sn$, set $A_a\coloneqq I+(p-2)a\otimes a$. Since
\[
  I\le A_a\le(p-1)I,
\]
equation \eqref{eq:Fq} implies that every normalized shifted solution satisfies
\[
  -\abs{Dv+q}^{p-2}\mathcal M^+(D^2v)\le0,
  \qquad
  -\abs{Dv+q}^{p-2}\mathcal M^-(D^2v)\ge0
  \quad\text{in }B_1,
\]
where $\mathcal M^\pm$ are the Pucci extremal operators with ellipticity constants $1$ and $p-1$. These inequalities hold in the full viscosity sense, since the factor vanishes at tests with $D\phi+q=0$. Applying the cutting lemma of \cite[Lemma~6]{ImbertSilvestre2013} separately to the subsolution and supersolution inequalities gives
\[
  -\mathcal M^+(D^2v)\le0,
  \qquad
  -\mathcal M^-(D^2v)\ge0
  \quad\text{in }B_1.
\]
Thus, every normalized shifted solution belongs to the same Pucci class, independently of $q$. The Krylov--Safonov interior H\"older estimate, see \cite{CaffarelliCabre1995}, therefore gives $\sigma=\sigma(n,p)\in(0,1)$ and $C=C(n,p)$ such that
\[
  [v]_{C^{0,\sigma}(B_{3/4})}\le C.
\]
Thus every $v\in\Kcal_0$ satisfies
\[
  \norm v_{C(\overline B_{1/2})}\le1,
  \qquad
  \abs{v(x)-v(y)}\le C\abs{x-y}^{\sigma}
  \quad\text{for all }x,y\in\overline B_{1/2},
\]
and this extends to every $h\in\Kcal$ by uniform convergence. Hence $\Kcal$ is uniformly bounded and equicontinuous on $\overline B_{1/2}$. By Arzel\`a--Ascoli, every sequence in $\Kcal$ has a subsequence converging uniformly on $\overline B_{1/2}$. Since $\Kcal$ is closed in $C(\overline B_{1/2})$ by definition, the limit belongs to $\Kcal$. This proves compactness in the uniform topology.

Fix $h\in\Kcal$. By the definition of $\Kcal$, there are sequences $q_j\in\Rn$ and $v_j\in C(B_1)$ such that
\[
  \Fcal_{q_j}(Dv_j,D^2v_j)=0\quad\text{in }B_1, \qquad \norm{v_j}_{L^\infty(B_1)}\le1,
\]
in the viscosity sense, and
\[
  \norm{v_j-h}_{C(\overline B_{1/2})}\longrightarrow0.
\]
Passing to a subsequence, we may assume that either $q_j\to q\in\Rn$ or $\abs{q_j}\to\infty$.
In the first case, $u_j\coloneqq v_j+q_j\cdot x$ is $p$-harmonic by \Cref{lem:equivalence}, and $u_j\to h+q\cdot x$ locally uniformly. The local-uniform stability of $p$-harmonic functions therefore shows that $h+q\cdot x$ is weakly $p$-harmonic.

In the second case, passing to a further subsequence, $q_j/\abs{q_j}\to e\in\Sn$. Hence
\[
  \Fcal_{q_j}(\xi,X)\longrightarrow
  -\operatorname{tr}\bigl((I+(p-2)e\otimes e)X\bigr)
\]
locally uniformly in $(\xi,X)\in\Rn\times\operatorname{Sym}(n)$. By viscosity stability, $h$ satisfies
\[
  -\operatorname{tr}\bigl((I+(p-2)e\otimes e)D^2h\bigr)=0
  \quad\text{in }B_{1/2}.
\]
This proves the alternatives.
\end{proof}

\subsection{Improvement of flatness on a variable scale}

The following flatness lemma is the step in which the qualitative regularity hypothesis of \Cref{thm:intro-principle} is used.

\begin{lemma}\label{lem:holder-flatness}
Assume the qualitative regularity hypothesis of \Cref{thm:intro-principle}, and let $0<\beta<\gamma$. Then there exists an integer $N\ge1$ and pairs
\[
(\rho_i,b_i)\in(0,1/4)\times\Rn,\qquad 1\le i\le N,
\]
such that, for every $q\in\Rn$ and every viscosity solution
\[
\Fcal_q(Dv,D^2v)=0\quad\text{in }B_1,\qquad \norm{v}_{L^\infty(B_1)}\le1,
\]
there is an index $i\in\{1,\ldots,N\}$ for which
\begin{equation}\label{eq:holder-one-step}
\norm{v-v(0)-b_i\cdot x}_{L^\infty(B_{\rho_i})}<\frac14\rho_i^{1+\beta}.
\end{equation}
\end{lemma}

\begin{proof}
We construct a relatively open covering of $\Kcal$. Fix $h\in\Kcal$. By \Cref{lem:profile-compactness} and the qualitative regularity hypothesis, $h\in C^{1,\gamma}_{\mathrm{loc}}(B_{1/2})$. Indeed, in the first alternative, it differs from a $p$-harmonic function by an affine function, while in the second alternative, it is smooth. Set
\[
\ell_h(x)\coloneqq h(0)+Dh(0)\cdot x.
\]
Since $\beta<\gamma$, we may choose $\rho_h\in(0,1/4)$ so small that
\begin{equation}\label{eq:holder-profile-strict}
\norm{h-\ell_h}_{L^\infty(B_{\rho_h})}\le[Dh]_{C^{0,\gamma}(B_{1/4})}\rho_h^{1+\gamma}<\frac18\rho_h^{1+\beta}.
\end{equation}
Define
\begin{equation}\label{eq:holder-open-set}
\mathcal O_h\coloneqq\left\{z\in\Kcal:\norm{z-\ell_h}_{L^\infty(B_{\rho_h})}<\frac18\rho_h^{1+\beta}\right\}.
\end{equation}
The map $z\mapsto\norm{z-\ell_h}_{L^\infty(B_{\rho_h})}$ is continuous in the uniform norm on $C(\overline B_{1/2})$, so $\mathcal O_h$ is relatively open in $\Kcal$. Moreover, \eqref{eq:holder-profile-strict} gives $h\in\mathcal O_h$. Thus $\{\mathcal O_h:h\in\Kcal\}$ is a relatively open covering of $\Kcal$.

By compactness, choose a finite subcovering $\mathcal O_{h_1},\ldots,\mathcal O_{h_N}$, and set
\[
\rho_i\coloneqq\rho_{h_i},\qquad b_i\coloneqq Dh_i(0).
\]
Let $v$ be any normalized shifted solution as in the statement. Its restriction to $\overline B_{1/2}$ belongs to $\Kcal_0\subset\Kcal$, and hence it lies in some $\mathcal O_{h_i}$. Since $\ell_{h_i}(0)=h_i(0)$, we obtain
\[
\begin{aligned}
\norm{v-v(0)-b_i\cdot x}_{L^\infty(B_{\rho_i})}
&\le\norm{v-\ell_{h_i}}_{L^\infty(B_{\rho_i})}+\abs{v(0)-h_i(0)}\\
&\le2\norm{v-\ell_{h_i}}_{L^\infty(B_{\rho_i})}
<\frac14\rho_i^{1+\beta}.
\end{aligned}
\]
This proves \eqref{eq:holder-one-step}.
\end{proof}

\subsection{Variable-scale iteration and the subendpoint estimate}

We record the exact affine rescaling before iterating. Take $\ell(y)=a+b\cdot y$, let $r,A>0$, and set
\begin{equation*}
  w(x)\coloneqq\frac{v(rx)-\ell(rx)}A.
\end{equation*}
If $v$ solves the shifted equation with parameter $q$, then $w$ solves it with parameter
\begin{equation}\label{eq:general-shift-update}
  \tilde{q}=\frac rA(q+b).
\end{equation}

\begin{proof}[Proof of \Cref{thm:intro-principle}]
Fix $0<\beta<\gamma$ and define $M\coloneqq\norm{u}_{L^\infty(B_{3/4})}$, which is finite by local boundedness. If $M=0$, the estimate is immediate. Otherwise, define
\begin{equation}\label{eq:scaling}
v(x)\coloneqq\frac{u(3x/4)-u(0)}{2M},\qquad x\in B_1.
\end{equation}
Then $v(0)=0$, $\norm{v}_{L^\infty(B_1)}\le1$, and \Cref{lem:equivalence} gives $\Fcal_0(Dv,D^2v)=0$ in $B_1$. 

We set $R_0=1$, $p_0=0$ and $P_0=0$, and construct inductively a sequence $\left( R_k,P_k\right)_k$, with 
$$P_k(x)=v(0)+p_k\cdot x,\qquad k=1,2, \ldots$$
such that
\begin{equation}\label{rachad}
\norm{v-P_k}_{L^\infty(B_{R_k})} \leq R_k^{1+\beta}, \qquad k\in \N_0.
\end{equation}
The case $k=0$ follows from our previous choices. Now, suppose $R_k$ and $P_k(x)$ have been constructed satisfying \eqref{rachad}. Define
\begin{equation*}
v_k(x)\coloneqq\frac{v(R_kx)-P_k(R_kx)}{R_k^{1+\beta}},
\end{equation*}
which satisfies $v_k(0)=0$ and, by \eqref{eq:general-shift-update},
\begin{equation*}
\Fcal_{q_k}(Dv_k,D^2v_k)=0\quad\text{in }B_1,\qquad \textrm{with} \quad  q_k=R_k^{-\beta}p_k.
\end{equation*}
Since we also have $\norm{v_k}_{L^\infty(B_1)}\le1$, we can apply \Cref{lem:holder-flatness}. It selects a pair $(r_k,b_k)=(\rho_{i_k},b_{i_k})$ for which
\begin{equation*}
\norm{v_k-b_k\cdot x}_{L^\infty(B_{r_k})}<\frac14r_k^{1+\beta}.
\end{equation*}
Defining $R_{k+1}\coloneqq R_kr_k$ and 
\begin{equation}\label{eq:holder-P-recursion}
P_{k+1}(y)\coloneqq P_k(y)+R_k^{\beta} b_k\cdot y,
\end{equation}
we obtain
\begin{equation*}
\norm{v-P_{k+1}}_{L^\infty(B_{R_{k+1}})}\le R_{k+1}^{1+\beta},
\end{equation*}
and the induction is complete.

We have
$$\theta^k\le R_k\le\Theta^k\longrightarrow0,$$
where $\theta=\min_i\rho_i$ and $\Theta=\max_i\rho_i$. The finitely many slopes $b_k$ are uniformly bounded. Indeed, let $h_i$ be the profiles selected in the proof of \Cref{lem:holder-flatness}, so that $\ell_{h_i}(x)=h_i(0)+b_i\cdot x$. By \eqref{eq:holder-profile-strict} and $\norm{h_i}_{C(\overline B_{1/2})}\le1$, it follows that
\[
\begin{aligned}
\rho_i\abs{b_i}
&=\norm{\ell_{h_i}-h_i(0)}_{L^\infty(B_{\rho_i})}\\
&\le\norm{\ell_{h_i}-h_i}_{L^\infty(B_{\rho_i})}
+\norm{h_i-h_i(0)}_{L^\infty(B_{\rho_i})}\\
&<\frac18\rho_i^{1+\beta}+2<3.
\end{aligned}
\]
Consequently,
\begin{equation*}
\abs{b_i}\le B_\theta\coloneqq\frac3\theta.
\end{equation*}
Taking gradients in \eqref{eq:holder-P-recursion} yields
\begin{equation*}
p_{k+1}-p_k=R_k^\beta b_k,
\end{equation*}
and thus the sequence $(p_k)$ converges to a vector $p_\infty$.
Since $R_{k+j}\le R_k\Theta^j$ and $\Theta<1/4$, 
\begin{equation}\label{eq:holder-slope-tail}
\abs{p_\infty-p_k}\le S_{\beta,\theta}R_k^\beta,
\qquad S_{\beta,\theta}\coloneqq\frac{3}{\theta(1-4^{-\beta})}.
\end{equation}

Finally, let $0<\abs x<\theta$ and, by the Archimedean property of the reals, choose $k\ge1$ such that 
$$R_{k+1}< \abs x \leq R_k.$$ 
Since $R_{k+1}\ge\theta R_k$, we have $R_k < \theta^{-1}\abs x$. From \eqref{rachad} and \eqref{eq:holder-slope-tail}, we have
\[
\begin{aligned}
\abs{v(x)-v(0)-p_\infty\cdot x}
&\le\abs{v(x)-P_k(x)}+\abs{p_k-p_\infty}\abs x\\
&\le(1+S_{\beta,\theta})R_k^{1+\beta}\\
&\le(1+S_{\beta,\theta})\theta^{-(1+\beta)}\abs x^{1+\beta}.
\end{aligned}
\]
It follows that $v$ is differentiable at the origin and $Dv(0)=p_\infty$. From \eqref{eq:holder-slope-tail} with $k=0$, we immediately obtain 
\begin{equation*}
\abs{Dv(0)}\le S_{\beta,\theta}.
\end{equation*}

For each $x_0\in B_{2/3}$, consider
\begin{equation*}
z_{x_0}(y)\coloneqq\frac{v(x_0+y/4)-v(x_0)}{2},\qquad y\in B_1.
\end{equation*}
Since $B_{1/4}(x_0)\subset B_1$, we have $\norm{z_{x_0}}_{L^\infty(B_1)}\le1$, and $z_{x_0}$ is $p$-harmonic in $B_1$. Applying the estimates to $z_{x_0}$ at the origin gives a uniform gradient bound and uniform affine remainders of order $1+\beta$ at every point in $B_{2/3}$. The standard characterization of $C^{1,\beta}$ by affine approximation gives a uniform $C^{1,\beta}$ estimate for $v$ in $B_{2/3}$. Returning to the original variables via \eqref{eq:scaling} yields \eqref{eq:intro-main-estimate}.
\end{proof}

\subsection{Planar estimates and optimal regularity thresholds}

The main theorem of Iwaniec and Manfredi in \cite{IwaniecManfredi1989} provides the qualitative regularity hypothesis of \Cref{thm:intro-principle}, with $\gamma=\alpha_{\mathrm{IM}}(p)$, where the exponent is defined in \eqref{eq:intro-IM-exponent}.

\begin{proof}[Proof of \Cref{cor:intro-planar}]
Fix $0<\beta<\alpha_{\mathrm{IM}}(p)$. By local boundedness,
\[
M\coloneqq\norm{u}_{L^\infty(B_{3/4})}\le C(p)\norm{u}_{L^p(B_1)}.
\]
Applying \Cref{thm:intro-principle} to $u$, with $\gamma=\alpha_{\mathrm{IM}}(p)$, and using $\norm{u}_{L^\infty(B_{1/2})}\le M$, gives
\[
\norm{u}_{C^{1,\beta}(B_{1/2})}\le C(p,\beta)M \le C(p,\beta)\norm{u}_{L^p(B_1)},
\]
proving \eqref{eq:intro-planar-Lp}.
\end{proof}

We next compare the optimal qualitative and quantitative exponents in arbitrary dimensions. Define
\begin{align*}
  \alpha_{\mathrm{qual}}(n,p)
  &\coloneqq\sup\left\{\alpha\in(0,1] \ : \    
  \text{every $p$-harmonic function in $B_1$ is locally in }C^{1,\alpha}
  \right\},\\
  \alpha_{\mathrm{quant}}(n,p)
  &\coloneqq\sup\left\{\alpha\in(0,1] \ : \ \begin{array}{l}
    \text{there exists $C_\alpha<\infty$ such that every $p$-harmonic function}\\[-0.1mm]
    \text{$u$ in $B_1$ satisfies \eqref{eq:intro-main-estimate}, with $\beta=\alpha$ and $C=C_\alpha$}
  \end{array}\right\}.
\end{align*}
Both sets are nonempty due to classical $p$-harmonic regularity theory.

\begin{corollary}\label{cor:thresholds}
For every $n\ge2$ and $p>2$,
\begin{equation*}
  \alpha_{\mathrm{qual}}(n,p)=\alpha_{\mathrm{quant}}(n,p).
\end{equation*}
\end{corollary}

\section{Uniform estimates for the planar infinity-Laplacian}
\label{sec:infinity-critical-points}

In this section, we prove \Cref{thm:planar-critical-almost-endpoint}. Infinity-harmonic functions are understood in the viscosity sense. The passage from qualitative regularity to a uniform estimate requires particular care for the infinity-Laplacian. Subtracting a nonconstant affine function changes the equation, and unbounded affine shifts can converge to an equation that controls only one direction. We instead rescale the solution about its value at the center. Compactness gives a uniform improvement while the normalized central gradient is small; once that gradient reaches a fixed threshold, a quantitative estimate near affine functions completes the argument.

\subsection{Regularity inputs and normalization}
\label{infq:inputs}

Write $B_r=B_r(0)$ and $Q_r=(-r,r)^2$. We work with the normalized class
\[
 \mathcal N=\bigl\{v\in C(B_1):\Delta_\infty v=0\text{ in }B_1,
 \ v(0)=0,\ \|v\|_{L^\infty(B_1)}\le1\bigr\}.
\]
By Evans--Savin \cite[Theorem~5.3]{EvansSavin2008}, there are universal constants $\alpha_0\in(0,1)$ and $C_0\ge1$ for which
\begin{equation}\label{infq:baseline}
 \|Dv\|_{L^\infty(B_{3/4})}
 +[Dv]_{C^{0,\alpha_0}(B_{3/4})}\le C_0,
 \quad v\in\mathcal N.
\end{equation}
We shall also use the following two facts:
\begin{enumerate}
\item Every planar infinity-harmonic function belongs locally to $C^{1,1/3}$ \cite[Theorem~1.1]{Xu2026}.
\item There are universal $\varepsilon_*>0$ and $M_*>0$ such that any infinity-harmonic function $w$ in $Q_4$ satisfies
\begin{equation}\label{infq:near-affine}
 \|w-x_1\|_{L^\infty(Q_4)}<\varepsilon_*
 \quad\Longrightarrow\quad
 [Dw]_{C^{0,1/3}(Q_1)}\le M_*.
\end{equation}
This is the case $\delta=1$ of \cite[Theorem~8.2]{Xu2026}.
\end{enumerate}
The first fact is used without a uniform bound for its H\"older seminorm.

Translations, orthogonal changes of coordinates, and the transformations
\begin{equation}\label{infq:invariance}
 v(x)\longmapsto \frac{v(z+rx)-c}{A},\qquad r,A>0,
\end{equation}
preserve infinity-harmonicity on the corresponding domains. Indeed, the transformed infinity-Laplacian is $r^4/A^3$ times the original one; the same identity for test functions gives the viscosity statement. All rescaled solutions below have this form.

\subsection{Improvement for small central gradients}

\begin{lemma}\label{infq:small-gradient}
Let $\beta\in(0,1/3)$. There exist $\eta\in(0,1)$ and finitely many radii $\rho_1,\ldots,\rho_N\in(0,1/8)$, depending only on $\beta$, such that every $v\in\mathcal N$ with $|Dv(0)|\le\eta$ satisfies
\begin{equation}\label{infq:improvement}
 \|v\|_{L^\infty(B_{\rho_i})}<\rho_i^{1+\beta},
 \qquad\text{for some }i\in\{1,\ldots,N\}.
\end{equation}
\end{lemma}

\begin{proof}
Let
\[
 \mathcal K_\infty
 =\overline{\{v|_{\overline B_{1/2}}:v\in\mathcal N\}}^{\,C^1(\overline B_{1/2})},
 \qquad
 \mathcal Z_\infty=\{h\in\mathcal K_\infty:Dh(0)=0\}.
\]
Estimate~\eqref{infq:baseline} gives uniform bounds and equicontinuity for the functions and their gradients on $\overline B_{1/2}$. Arzel\`a--Ascoli therefore makes $\mathcal K_\infty$ compact in $C^1(\overline B_{1/2})$. The map $h\mapsto Dh(0)$ is continuous, so $\mathcal Z_\infty$ is compact as well. Every $h\in\mathcal K_\infty$ satisfies $h(0)=0$ and is infinity-harmonic in $B_{1/2}$, by the stability of viscosity solutions for the continuous operator $(\xi,X)\mapsto\langle X\xi,\xi\rangle$; see \cite[Section~6]{CrandallIshiiLions1992}.

For $h\in\mathcal Z_\infty$, qualitative regularity gives 
\[
   |h(x)| \le  H_h|x|^{4/3},\qquad x\in B_{1/4},
\]
where \(H_h=[Dh]_{C^{0,1/3}(B_{1/4})}\), and we used that $h(0) = |Dh(0)| = 0$. The strict inequality $\beta<1/3$ allows us to choose $\rho_h\in(0,1/8)$ such that
\[
 \|h\|_{L^\infty(B_{\rho_h})}
 <\tfrac12\rho_h^{1+\beta}.
\]
Thus, the relatively open sets
\[
 O_h=\{z\in\mathcal K_\infty:
          \|z\|_{L^\infty(B_{\rho_h})}<\rho_h^{1+\beta}\}
\]
cover $\mathcal Z_\infty$. Choose a finite subcover $O_{h_1},\ldots,O_{h_N}$ and put $O=\bigcup_{i=1}^N O_{h_i}$ and $\rho_i=\rho_{h_i}$.

If $\mathcal K_\infty\setminus O$ is nonempty, compactness and its disjointness from $\mathcal Z_\infty$ imply
\[
 m:=\min_{z\in\mathcal K_\infty\setminus O}|Dz(0)|>0.
\]
Set $\eta=\min\{1/2,m/2\}$. If the complement is empty, set $\eta=1/2$. In either case, $v\in\mathcal N$ and $|Dv(0)|\le\eta$ imply $v|_{\overline B_{1/2}}\in O$, which is precisely~\eqref{infq:improvement}.
\end{proof}

The ambient compact set contains all normalized solutions. Selecting its subset with zero central gradient only after taking the closure is what permits the conclusion for small nonzero gradients.

\subsection{Uniform control above a fixed gradient threshold}

\begin{lemma}\label{infq:noncritical}
For each $\eta\in(0,1)$ there are $\lambda_\eta\in(0,1/16)$ and $A_\eta<\infty$, depending only on $\eta$, such that every $v\in\mathcal N$ with $|Dv(0)|\ge\eta$ satisfies
\begin{equation}\label{infq:noncritical-estimate}
 |v(y)-Dv(0)\cdot y|\le A_\eta|y|^{4/3},
 \quad y\in B_{\lambda_\eta}.
\end{equation}
\end{lemma}

\begin{proof}
Let $a=|Dv(0)|$. By~\eqref{infq:baseline}, $\eta\le a\le C_0$. Choose an orthogonal matrix $T$ with $Te_1=Dv(0)/a$, and set
\[
 L_0=C_0(4\sqrt2)^{1+\alpha_0},\qquad
 \lambda_\eta=
 \min\left\{\frac1{32},
       \left(\frac{\varepsilon_*\eta}{2L_0}\right)^{1/\alpha_0}\right\}.
\]
For $x\in Q_4$, define
\[
 w(x)=\frac{v(\lambda_\eta Tx)}{a\lambda_\eta}.
\]
The inclusion $\lambda_\eta TQ_4\subset B_{1/2}$ follows from $4\sqrt2\lambda_\eta<1/2$. Hence $w$ is infinity-harmonic in $Q_4$, and $w(0)=0$, $Dw(0)=e_1$. Applying~\eqref{infq:baseline} along segments gives
\[
 |v(y)-Dv(0)\cdot y|
 \le\frac{C_0}{1+\alpha_0}|y|^{1+\alpha_0},
 \quad y\in B_{1/2}.
\]
Consequently,
\[
 \|w-x_1\|_{L^\infty(Q_4)}
 \le\frac{L_0}{\eta}\lambda_\eta^{\alpha_0}
 \le\frac{\varepsilon_*}{2}<\varepsilon_*.
\]
By~\eqref{infq:near-affine}, $[Dw]_{C^{0,1/3}(Q_1)}\le M_*$. Since $B_1\subset Q_1$ and
\[
 Dw(x)=a^{-1}T^{\mathsf T}Dv(\lambda_\eta Tx),
\]
we obtain
\[
 [Dv]_{C^{0,1/3}(B_{\lambda_\eta})}
 \le aM_*\lambda_\eta^{-1/3}
 \le C_0M_*\lambda_\eta^{-1/3}.
\]
This proves~\eqref{infq:noncritical-estimate}, with $A_\eta=C_0M_*\lambda_\eta^{-1/3}$.
\end{proof}

\subsection{The stopping argument and the gradient estimate}

\begin{proof}[Proof of \Cref{thm:planar-critical-almost-endpoint}]
Fix $\beta\in(0,1/3)$ and the constants from \Cref{infq:small-gradient}. Set
\[
 \theta=\min_i\rho_i>0,\qquad
 \Theta=\max_i\rho_i<1/8,\qquad
 \lambda=\lambda_\eta,\qquad A=A_\eta.
\]
We first establish a constant $D_\beta$, independent of $v\in\mathcal N$, such that
\begin{equation}\label{infq:taylor}
 \|v-Dv(0)\cdot x\|_{L^\infty(B_r)}
 \le D_\beta r^{1+\beta},\qquad 0<r<1.
\end{equation}

\smallskip
\noindent\emph{Construction of the scales.}
Write $g=Dv(0)$ and start with $R_0=1$, $v_0=v$. As long as $|Dv_k(0)|<\eta$, apply \Cref{infq:small-gradient} to
\[
 v_k(x)=\frac{v(R_kx)}{R_k^{1+\beta}}\in\mathcal N.
\]
It selects $r_k\in\{\rho_1,\ldots,\rho_N\}$ for which
\[
 \|v_k\|_{L^\infty(B_{r_k})}<r_k^{1+\beta}.
\]
With $R_{k+1}=R_kr_k$, invariance~\eqref{infq:invariance} and this inequality imply
\[
 v_{k+1}(x)=\frac{v_k(r_kx)}{r_k^{1+\beta}}
          =\frac{v(R_{k+1}x)}{R_{k+1}^{1+\beta}}\in\mathcal N.
\]
At every constructed scale,
\begin{equation}\label{infq:scale-relations}
 Dv_k(0)=R_k^{-\beta}g,\qquad
 \|v\|_{L^\infty(B_{R_k})}\le R_k^{1+\beta},\qquad
 \theta R_k\le R_{k+1}\le\Theta R_k,
\end{equation}
where the last inequalities apply when the next scale is constructed.

If $g=0$, the iteration never stops and $R_k\le\Theta^k\to0$. For $0<r<1$, choose $j$ with $R_{j+1}\le r\le R_j$. Then $R_j\le r/\theta$, so
\[
 \|v\|_{L^\infty(B_r)}\le R_j^{1+\beta}
 \le\theta^{-(1+\beta)}r^{1+\beta}.
\]
This proves~\eqref{infq:taylor} when $g=0$.

Suppose now that $g\ne0$. The iteration must stop at a finite index $K$: otherwise $R_k\le\Theta^k$ and $|Dv_k(0)|=R_k^{-\beta}|g|\to\infty$, contradicting $|Dv_k(0)|<\eta$. At this first stopping index,
\begin{equation}\label{infq:stopping-gradient}
 \eta\le |Dv_K(0)|\le C_0,
 \qquad |g|<\eta R_j^\beta, \quad 0\le j<K.
\end{equation}
Put $R_*=R_K$. The index $K=0$ is allowed.

\smallskip
\noindent\emph{Radii above the stopping scale.}
If $R_*<r<1$, choose $j<K$ with $R_{j+1}\le r\le R_j$. By~\eqref{infq:scale-relations} and~\eqref{infq:stopping-gradient},
\begin{equation}\label{infq:outer-scales}
 \begin{aligned}
 \|v-g\cdot x\|_{L^\infty(B_r)}
 &\le R_j^{1+\beta}+|g|r\\
 &\le(1+\eta)R_j^{1+\beta}
 \le(1+\eta)\theta^{-(1+\beta)}r^{1+\beta}.
 \end{aligned}
\end{equation}
This range is empty when $K=0$.

\smallskip
\noindent\emph{Radii comparable to the stopping scale.}
For $\lambda R_*\le r\le R_*$, the bounds $\|v_K\|_{L^\infty(B_1)}\le1$ and $|Dv_K(0)|\le C_0$ give
\begin{equation}\label{infq:middle-scales}
 \begin{aligned}
 \|v-g\cdot x\|_{L^\infty(B_r)}
 &\le (1+C_0)R_*^{1+\beta}\\
 &\le (1+C_0)\lambda^{-(1+\beta)}r^{1+\beta}.
 \end{aligned}
\end{equation}

\smallskip
\noindent\emph{Radii below the noncritical scale.}
For $0<r\le\lambda R_*$, apply \Cref{infq:noncritical} to $v_K$. Using $Dv_K(0)=R_*^{-\beta}g$, we find
\begin{equation}\label{infq:inner-scales}
 \begin{aligned}
 \|v-g\cdot x\|_{L^\infty(B_r)}
 &=R_*^{1+\beta}
   \|v_K-Dv_K(0)\cdot x\|_{L^\infty(B_{r/R_*})}\\
 &\le A R_*^{1+\beta}(r/R_*)^{4/3}\\
 &=A r^{1+\beta}(r/R_*)^{1/3-\beta}
 \le A r^{1+\beta}.
 \end{aligned}
\end{equation}
The last inequality uses $r/R_*\le\lambda<1$ and $\beta<1/3$. Thus, no constant depends on $K$ or on $R_*$. Combining the three ranges with the case $g=0$ proves~\eqref{infq:taylor}, with
\begin{equation}\label{infq:taylor-constant}
 D_\beta=
 \max\bigl\{(1+\eta)\theta^{-(1+\beta)},
            (1+C_0)\lambda^{-(1+\beta)},A\bigr\}.
\end{equation}

\smallskip
\noindent\emph{Recentring and comparison of tangent planes.}
Consider first a solution $u$ in $B_1$, and put $M=\osc_{B_{3/4}}u$, which is finite by continuity on the compact set $\overline B_{3/4}\subset B_1$. If $M=0$, then $u$ is constant on $B_{3/4}$ and the conclusion is immediate. Otherwise, for each $z\in B_{1/2}$ the function
\[
 v_z(x)=\frac{u(z+x/4)-u(z)}{M},\qquad x\in B_1,
\]
belongs to $\mathcal N$, since $B_{1/4}(z)\subset B_{3/4}$. Let $L_z(y)=u(z)+Du(z)\cdot(y-z)$. Rescaling~\eqref{infq:taylor} gives
\begin{equation}\label{infq:uniform-expansions}
 |u(y)-L_z(y)|\le E_\beta M|y-z|^{1+\beta}
 \quad\text{if }|y-z|<1/4,
 \qquad E_\beta=4^{1+\beta}D_\beta.
\end{equation}
For a point on the boundary of a chosen comparison ball, the same bound follows by continuity. Also, $Dv_z(0)=Du(z)/(4M)$ and~\eqref{infq:baseline} imply
\begin{equation}\label{infq:gradient-bound}
 |Du(z)|\le4C_0M, \quad z\in B_{1/2}.
\end{equation}

Let $z,z'\in B_{1/2}$ and $0<d=|z-z'|<1/16$. Set $m=(z+z')/2$. For $y\in B_d(m)$, both $|y-z|$ and $|y-z'|$ are at most $3d/2<1/4$. Hence
\[
 \|L_z-L_{z'}\|_{L^\infty(B_d(m))}
 \le 2E_\beta M(3d/2)^{1+\beta}.
\]
If $P(y)=b+q\cdot(y-m)$ is affine, then
\[
 2d|q|=\osc_{B_d(m)}P
 \le2\|P\|_{L^\infty(B_d(m))}.
\]
Applying this identity to $P=L_z-L_{z'}$ yields
\[
 |Du(z)-Du(z')|
 \le2(3/2)^{1+\beta}E_\beta M d^\beta.
\]
For $d\ge1/16$, estimate~\eqref{infq:gradient-bound} instead gives
\[
 |Du(z)-Du(z')|\le8C_0M
 \le8C_0\,16^\beta M d^\beta.
\]
We have proved~\eqref{eq:planar-uniform-gradient}, for instance with
\[
 C_\beta=4C_0+
 \max\bigl\{2(3/2)^{1+\beta}E_\beta,\,8C_0\,16^\beta\bigr\}.
\]
For a solution in $B_R(x_0)$, applying this estimate to $x\mapsto u(x_0+Rx)$ gives
\begin{equation}\label{infq:scaled-estimate}
 R\|Du\|_{L^\infty(B_{R/2}(x_0))}
 +R^{1+\beta}[Du]_{C^{0,\beta}(B_{R/2}(x_0))}
 \le C_\beta\osc_{B_{3R/4}(x_0)}u.
\end{equation}
\end{proof}

\begin{remark}[Critical-point decay]
\label{infq:scope}
If $Du(x_0)=0$, estimate~\eqref{infq:scaled-estimate} and integration along segments give, for $0<r\le R/2$,
\begin{equation}\label{eq:planar-critical-main-estimate}
 \|u-u(x_0)\|_{L^\infty(B_r(x_0))}
 +r\|Du\|_{L^\infty(B_r(x_0))}
 \le C_\beta(r/R)^{1+\beta}\osc_{B_{3R/4}(x_0)}u,
\end{equation}
after enlarging $C_\beta$.
\end{remark}

\section{The nonlocal superconductivity problem}\label{sec:nonlocal}

In this final section, we prove \Cref{thm:nonlocal-conditional}. Let $s\in(1/2,1)$. The homogeneous nonlocal superconductivity problem studied in \cite{AraujoSobral2026} is 
\[
(-\Delta)^su=0\quad\text{in }B_1\cap\Omega,
\qquad \abs{Du}=0\quad\text{in }B_1\setminus\Omega,
\]
where the unknowns are the function $u$ and the open set $\Omega$, and 
\[
(-\Delta)^s u(x)=c_{n,s}\,\operatorname{P.V.}\int_{\Rn}\frac{u(x)-u(y)}{\abs{x-y}^{n+2s}}\,dy.
\]
We work with the associated nonlocal equation \eqref{eq:scalar-nonlocal}, which is satisfied in the viscosity sense by every solution of the problem. As in the $p$-harmonic case, we first consider the shifted equation generated by affine subtraction. The approximation argument also establishes the growth bound at infinity needed for the iteration.

\subsection{The shifted equation and local estimates}

For $q\in\Rn$, we consider
\begin{equation}\label{eq:nonlocal-shifted}
\abs{Dv+q}(-\Delta)^sv=0\quad\text{in }B_1,
\end{equation}
whose solutions belong to $C(B_1)\cap L_s^1(\Rn)$ (recall \eqref{eq:tail-space}) and are understood in the viscosity sense. More precisely, if $\phi\in C^2$ touches $v$ from above at $x_0\in B_1$, let $\phi_{\delta,v}$ agree with $\phi$ in $B_\delta(x_0)$ and with $v$ outside that ball. For every sufficiently small $\delta>0$, the required inequality is
\[
\abs{D\phi(x_0)+q}(-\Delta)^s\phi_{\delta,v}(x_0)\le0.
\]
The reverse inequality holds for tests from below. In particular, no condition is imposed when $D\phi(x_0)+q=0$. Exterior values need only be measurable, and inequalities concerning those values are understood almost everywhere.

Since $s>1/2$, affine functions belong to the space $L_s^1(\Rn)$  and have vanishing fractional Laplacian. Thus $v$ solves \eqref{eq:nonlocal-shifted} if and only if $v+q\cdot x$ solves the nonlocal equation \eqref{eq:scalar-nonlocal}. The differentiability hypothesis of \Cref{thm:nonlocal-conditional} therefore applies to solutions of the shifted equation after adding an affine function.

For solutions bounded in $B_1$, we use the local Lipschitz estimate
\begin{equation}\label{eq:nonlocal-local-lipschitz}
[v]_{C^{0,1}(B_{1/2})}\le C\left(\norm{v}_{L^\infty(B_1)}+\norm{v}_{L_s^1(\Rn)}\right),
\end{equation}
where $C=C(n,s)$ is independent of $q$. This follows from the localized proof of \cite[Proposition~2.1 and Lemmas~2.2--2.3]{DosPrazeresTopp2021}. The same argument applies to the nonhomogeneous equation obtained after a cutoff, yielding \eqref{eq:nonlocal-local-lipschitz} with a constant independent of $q$; see \cite[proof of Theorem~1.2]{AraujoSobral2026}.

\subsection{Approximation and growth at infinity}

Fix $\alpha_1\coloneqq(2s-1)/2$. Since $1+\alpha_1<2s$, the bound
\[
\abs{v(x)}\le\max\{1,\abs x^{1+\alpha_1}\}, \quad x\in\Rn,
\]
gives a uniform bound for $\norm{v}_{L_s^1(\Rn)}$. Hence \eqref{eq:nonlocal-local-lipschitz} gives \([v]_{C^{0,1}(B_{1/2})}\le L\) for some $L=L(n,s)\ge1$, independently of the shift.

\begin{lemma}\label{lem:nonlocal-approximation}
Assume the differentiability hypothesis of \Cref{thm:nonlocal-conditional}. There exist $K\ge2$ and $0<\theta<\rho_*<1/(4K)$, depending only on $n$ and $s$, with the following property: for every $q\in\Rn$ and every solution $v$ of \eqref{eq:nonlocal-shifted} satisfying
\[
\abs{v(x)}\le\max\{1,\abs x^{1+\alpha_1}\}, \quad x\in\Rn,
\]
there exist $r\in[\theta,\rho_*]$ and $b\in\Rn$, with $\abs b\le L$, such that
\begin{equation}\label{eq:nonlocal-affine-approximation}
\norm{v-v(0)-b\cdot x}_{L^\infty(B_{Kr})}<\frac r8.
\end{equation}
Moreover, the rescaled function
\[
w(x)\coloneqq\frac8r\bigl(v(rx)-v(0)-rb\cdot x\bigr)
\]
satisfies \(\abs{w(x)}\le\max\{1,\abs x^{1+\alpha_1}\}\), for \(x \in \Rn\).
\end{lemma}

\begin{proof}
We construct an open covering of a compact family of global solutions. Choose $K\ge2$ and $\rho_*\in(0,1/(4K))$ so that
\begin{equation}\label{eq:nonlocal-choice-of-scales}
K^{\alpha_1}\ge16L,
\qquad 8\cdot2^{\alpha_1}(4+L)\rho_*^{\alpha_1}\le1.
\end{equation}
Define
\begin{equation*}
  \Kcal \coloneqq
  \left\{
    h \colon \Rn \to \mathbb{R} \ : \
    \begin{array}{l}
      h(0)=0, \quad [h]_{C^{0,1}(\Rn)}\le L, \\[2pt]
      \abs{Dh+q}(-\Delta)^s h=0 \ \text{in } \Rn,
      \quad \text{for some } q\in\Rn
    \end{array}
  \right\}.
\end{equation*}
This family is compact in the locally uniform topology of $\Rn$. Indeed, the bound $\abs{h(x)}\le L\abs x$ gives subsequential convergence by Arzel\`a--Ascoli, and dominated convergence gives convergence in $L_s^1(\Rn)$, since $2s>1$. For a sequence with shifts $q_j$, pass to a further subsequence such that either $q_j\to q$ or $\abs{q_j}\to\infty$. In the first case, stability gives the shifted equation for the limit. In the second case, dividing the viscosity inequalities by the nonzero shifted gradient factors gives $(-\Delta)^sh=0$ in the limit, which belongs to $\Kcal$ with shift zero, see \cite[Lemma~4.3]{CaffarelliSilvestre2011}. Moreover, every $h\in\Kcal$ is differentiable at the origin by the hypothesis applied to $h+q\cdot x$, and $\abs{Dh(0)}\le L$.

For each $h\in\Kcal$, choose $\rho_h\in(0,\rho_*)$ such that
\[
\norm{h-Dh(0)\cdot x}_{L^\infty(B_{K\rho_h})}<\frac{\rho_h}{8},
\]
and define the open set
\[
\mathcal O_h\coloneqq\left\{z\in C(\overline B_{1/2}):\norm{z-Dh(0)\cdot x}_{L^\infty(B_{K\rho_h})}<\frac{\rho_h}{8}\right\}.
\]
These sets cover the restrictions of $\Kcal$ to $\overline B_{1/2}$. By compactness, we extract a finite subcovering $\mathcal O_{h_1},\ldots,\mathcal O_{h_N}$, with corresponding radii $\rho_1,\ldots,\rho_N$.

We claim that there is a sufficiently small $t_0\in(0,1/2)$ such that, for every solution $v$ in the statement of the lemma, the restriction of $x\mapsto(v(t_0x)-v(0))/t_0$ to $\overline B_{1/2}$ belongs to one of the sets in this finite subcovering. Suppose otherwise. Then, there exist parameters $t_j\in(0,1/2)$ with $t_j\downarrow0$ and solutions $v_j$ as in the statement of the lemma, with shifts $q_j$, such that the functions
\[
u_j(x)\coloneqq\frac{v_j(t_jx)-v_j(0)}{t_j}
\]
satisfy
\[
u_j|_{\overline B_{1/2}}\notin\bigcup_{i=1}^N\mathcal O_{h_i}.
\]
We have $u_j(0)=0$ and
\[
\abs{Du_j+q_j}(-\Delta)^su_j=0\quad\text{in }B_{1/t_j},
\qquad [u_j]_{C^{0,1}(B_{1/(2t_j)})}\le L.
\]
Using the Lipschitz estimate for $t_j\abs x\le1/2$, and the growth bound for $t_j\abs x > 1/2$, gives
\[
\abs{u_j(x)}\le(L+4)\abs x+\abs x^{1+\alpha_1}, \quad x\in\Rn.
\]
Since $1+\alpha_1<2s$, Arzel\`a--Ascoli on expanding balls and dominated convergence give, after passing to a subsequence,
\[
u_j\longrightarrow h\quad\text{locally uniformly in }\Rn,
\qquad \norm{u_j-h}_{L_s^1(\Rn)}\longrightarrow0.
\]
The preceding stability argument shows that $h\in\Kcal$. Its restriction to $\overline B_{1/2}$ therefore belongs to some $\mathcal O_{h_i}$. By openness and uniform convergence, the same holds for $u_j$, for all sufficiently large $j$, a contradiction. This proves the claim.

Set $\theta\coloneqq t_0\min_i\rho_i$. For a solution $v$ as in the statement of the lemma, choose the element of the subcovering and the corresponding $\rho_i$, and take $r\coloneqq t_0\rho_i$ and $b\coloneqq Dh_i(0)$. Membership in $\mathcal O_{h_i}$ and rescaling give \eqref{eq:nonlocal-affine-approximation}, with $r\in[\theta,\rho_*]$ and $\abs b\le L$.

It remains to verify the growth of $w$. For $\abs x\le K$, \eqref{eq:nonlocal-affine-approximation} and continuity give $\abs{w(x)}\le1$. If $\abs x>K$ and $r\abs x\le1/2$, then
\[
\abs{w(x)}\le16L\abs x\le\abs x^{1+\alpha_1}.
\]
If $r\abs x\ge1/2$, the original growth bound gives
\[
\begin{aligned}
\abs{v(rx)-v(0)-rb\cdot x}
&\le\max\{1,(r\abs x)^{1+\alpha_1}\}+1+Lr\abs x\\
&\le2^{\alpha_1}(4+L)(r\abs x)^{1+\alpha_1}.
\end{aligned}
\]
Thus, by \eqref{eq:nonlocal-choice-of-scales},
\[
\abs{w(x)}\le8\cdot2^{\alpha_1}(4+L)r^{\alpha_1}\abs x^{1+\alpha_1}\le\abs x^{1+\alpha_1}.
\]
\end{proof}

\subsection{Proof of \Cref{thm:nonlocal-conditional}}

Let $M\coloneqq\norm{u}_{L^\infty(\Rn)}$. If $M=0$, the estimate is immediate. Otherwise, $v\coloneqq u/M$ solves \eqref{eq:scalar-nonlocal} and satisfies $\abs v\le1$ in $\Rn$.

Starting with $R_0=1$ and $P_0=0$, apply \Cref{lem:nonlocal-approximation} iteratively to
\[
v_k(x)\coloneqq\frac{v(R_kx)-P_k(R_kx)}{8^{-k}R_k},
\qquad q_k\coloneqq8^kp_k,
\]
where $p_k=DP_k$. At each step, choose $r_k\in[\theta,\rho_*]$ and $\abs{b_k}\le L$ as in the lemma, and set
\[
R_{k+1}\coloneqq R_kr_k,
\qquad P_{k+1}(x)\coloneqq v(0)+(p_k+8^{-k}b_k)\cdot x.
\]
Rescaling and \Cref{lem:nonlocal-approximation} ensure that $v_{k+1}$ solves the shifted equation and satisfies $|v_{k+1}|\le\max\{1,\abs x^{1+\alpha_1}\}$ throughout $\Rn$, so the construction continues for every $k$. We have $\theta^k\le R_k\le\rho_*^k\to0$, $P_k(0)=v(0)$ for $k\ge1$, and
\begin{equation}\label{eq:nonlocal-iterated-flatness}
\norm{v-P_k}_{L^\infty(B_{R_k})}\le8^{-k}R_k,
\qquad p_{k+1}-p_k=8^{-k}b_k.
\end{equation}
Thus, $p_k\to p_\infty$, with
\begin{equation}\label{eq:nonlocal-slope-tail}
\abs{p_\infty-p_k}\le\frac{8L}{7}\,8^{-k},
\qquad \abs{p_\infty}\le\frac{8L}{7}.
\end{equation}

Choose $\alpha=\alpha(n,s)\in(0,\alpha_1)$ such that $1/8\le\theta^\alpha$. Then $8^{-k}\le R_k^\alpha$. For $0<\abs x<\theta$, choose $k\ge1$ with $R_{k+1} < \abs x \leq R_k$. Since $R_k\le\theta^{-1}\abs x$, \eqref{eq:nonlocal-iterated-flatness} and \eqref{eq:nonlocal-slope-tail} give
\[
\abs{v(x)-v(0)-p_\infty\cdot x}
\le\left(1+\frac{8L}{7}\right)8^{-k}R_k
\le C\abs x^{1+\alpha}.
\]
Consequently, $Dv(0)=p_\infty$ and $\abs{Dv(0)}\le8L/7$.

Apply these estimates to the globally bounded solutions
\[
z_{x_0}(y)\coloneqq\frac{v(x_0+y/4)-v(x_0)}2,
\quad y\in\Rn,\quad x_0\in B_{1/2}.
\]
Each $z_{x_0}$ solves \eqref{eq:scalar-nonlocal} in $B_1$ and satisfies $\abs{z_{x_0}}\le1$ in $\Rn$. Rescaling and the characterization of $C^{1,\alpha}$ by affine approximation yield
\[
\norm{Dv}_{L^\infty(B_{1/2})}+[Dv]_{C^{0,\alpha}(B_{1/2})}\le C(n,s).
\]
Multiplying by $M$ proves \eqref{eq:nonlocal-conditional-estimate}.
\qed

\bigskip

{\small \noindent{\bf Acknowledgments.} This publication is based upon work supported by King Abdullah University of Science and Technology (KAUST) under Award No. ORFS-CRG12-2024-6430. ET acknowledges partial support from the Grayce B. Kerr Chair funds at Oklahoma State University.}

\bigskip

{\small \noindent{\bf AI disclosure.} The mathematical ideas, statements, and proofs were developed entirely by the authors. AI tools were solely used to correct typographical errors and improve the language.}

\end{document}